\documentclass[12pt]{article}

\usepackage[inner=32mm,outer=31mm,tmargin=30mm,bmargin=40mm]{geometry}
\usepackage[normalem]{ulem}
\usepackage{latexsym,amsfonts,amsmath,amssymb,epsfig,tabularx,amsthm,dsfont,mathrsfs}

\usepackage{graphicx}
\usepackage{enumerate}

\usepackage{booktabs}
\usepackage{titling}
\newcommand{\subtitle}[1]{%
  \posttitle{%
    \par\end{center}
    \begin{center}\large#1\end{center}
    \vskip0.5em}%
}

\usepackage{hyperref} 
\hypersetup{colorlinks=true,
        linkcolor=black,
        citecolor=black,
        urlcolor=blue}

\usepackage{pgfplots}

\theoremstyle{plain}

\newtheorem{theorem}{Theorem}[section]

\newtheorem{corollary}[theorem]{Corollary}

\theoremstyle{definition}
\newtheorem{remark}[theorem]{Remark}
\newtheorem{example}[theorem]{Example}
\newtheorem{assumption}[theorem]{Assumption}

\newcommand{\dint}{\mathup{d}}

\renewcommand{\rho}{\varrho}

\DeclareMathAlphabet{\mathup}{OT1}{\familydefault}{m}{n}

\newcommand{\wt}{\widetilde}
\newcommand{\widebar}[1]{\mbox{\kern1.5pt\hbox{\vbox{\hrule height 0.6pt \kern0.35ex
        \hbox{\kern-0.15em \ensuremath{#1 }\kern0.0em}}}}\kern-0.1pt}

\newlength{\fixboxwidth}
\usepackage{color}
\usepackage{soul}

\usepackage{latexsym,amsfonts,amsmath,amssymb,mathrsfs}
\usepackage{url}
\usepackage{graphicx}
\usepackage{color}
\usepackage{euscript}
\usepackage{verbatim}
\usepackage{hyperref}

\definecolor{owngreen}{rgb}{0, 0.7, 0.2}

\usepackage[textsize=tiny]{todonotes}	

\begin{document}

\title{Stability 
	of doubly-intractable distributions}

\author{
	Michael Habeck\thanks{Friedrich-Schiller-Universit\"at Jena, Max-Planck-Institut f\"ur biophysikalische Chemie,
		Arbeitsgruppe Mikroskopische Bildanalyse, Kollegiengasse 10, 07743 Jena, Germany,
		Email: michael.habeck@uni-jena.de}, 
	Daniel Rudolf\thanks{Institute for Mathematical Stochastics, 
		University of Goettingen, Goldschmidtstra\ss e 7, 37077 G\"ottingen, Germany, 
		Email: daniel.rudolf@uni-goettingen.de},
	Bj\"orn Sprungk\thanks{Faculty of Mathematics and Computer Science, Technische
		Universit\"at Bergakademie Freiberg, 09596 Freiberg, Germany,
		Email: bjoern.sprungk@math.tu-freiberg.de}
}

\date{\today}

\maketitle
\begin{abstract}
Doubly-intractable distributions appear naturally as posterior distributions in Bayesian inference frameworks whenever the likelihood contains a normalizing function $Z$. Having two such functions $Z$ and $\wt Z$ we provide estimates of the total variation and $1$-Wasserstein distance of the resulting posterior probability measures. As a consequence this leads to local Lipschitz continuity w.r.t.~$Z$. In the more general framework of a random function $\wt Z$ we derive bounds on the expected total variation and expected $1$-Wasserstein distance. The applicability of the estimates is illustrated within the setting of two representative Monte Carlo recovery scenarios.
\end{abstract}


{\bf Classification. Primary: 60B10; Secondary: 62C10, 65C05.}

\section{Introduction}	
Suppose that $(\Theta,d)$ is a complete and separable metric space equipped with its Borel $\sigma$-algebra $\mathcal{B}(\Theta)$ and a $\sigma$-finite reference measure $\mu$.
Assume that there are measurable functions $\Phi \colon \Theta \to (-\infty,\infty)$ and $Z\colon \Theta \to (0,\infty)$ with
\[
C_Z := \int_\Theta \frac{\exp(-\Phi(\theta))}{Z(\theta)} \mu({\rm d}\theta) \in(0,\infty), 
\]
such that 
\begin{equation}
\label{eq: pi_Z}
\pi_{Z}(A) := \frac{1}{C_Z} \int_A \frac{\exp(-\Phi(\theta))}{Z(\theta)} \mu({\rm d}\theta),
\qquad A\in\mathcal{B}(\Theta),
\end{equation}
defines a probability measure on $(\Theta,\mathcal{B}(\Theta))$. 
We are interested in stability properties of $\pi_Z$ w.r.t.~the function $Z$. Given another measurable function $\widetilde{Z} \colon \Theta \to (0,\infty)$ which is somehow close to $Z$, we ask whether $\pi_Z$ and $\pi_{\widetilde{Z}}$ are also close to each other.  

This question is motivated by applications within Bayesian inference, where such type of distributions appear, see for example \cite{habeck2014bayesian,hunter2006inference} as well as \cite{park2018bayesian} and the references therein.
The interpretation is as follows: Think of $\theta \mapsto \exp(-\Phi(\theta))/Z(\theta)$ as a likelihood function which contains an unknown normalizing function $Z$. One is interested on sampling w.r.t.~a posterior distribution based on the partially unknown likelihood function. 
Unknown here in the sense that it is infeasible to evaluate $Z$ exactly. This and the fact that \eqref{eq: pi_Z} itself contains an unknown normalizing constant $C_Z$ 
is the reason for calling $\pi_Z$ doubly-intractable. A recent survey for approximate sampling of such doubly-intractable distributions is given in \cite{park2018bayesian}.
We provide a motivating example for such scenarios.
\begin{example}[Gibbs distribution as likelihood]
	\label{ex: Boltzmann}
	A Gibbs distribution on a finite state space $G$ is determined by a probability mass function $\rho(\cdot\mid \beta)$ with inverse temperature parameter $\beta>0$ given by
	\[
	\rho(x\mid \beta) = \exp(-\beta H(x))/Z(\beta), \quad x\in G,
	\]
	where $H\colon G \to [0,\infty)$ is called Hamiltonian  and $Z(\beta) = \sum_{x\in G} \exp(-\beta H(x))$ partition function. 
	Suppose that there is observational data $x_{\rm obs}\in G$ available as a realization of the Gibbs distribution but with unknown $\beta$.
	We then aim to gain knowledge of $\beta$ through the realization of $x_{\rm obs}$. In a Bayesian framework this leads to a posterior distribution of the form \eqref{eq: pi_Z} with $\Theta = (0,\infty)$ and 
	$\Phi(\beta)=\beta H(x_{\rm obs})$ and $Z(\beta)$. 
	Note that the Ising model fits into this framework: Let $(V,E)$ be a graph with (non-empty) vertex set $E$, edge set $V\subset E\times E$ and $G=\{-1,1\}^{E}$ with $H(x) = -\sum_{(e,e')\in V} x(e)x(e')$.
\end{example}	

Having such an example in mind it is reasonable to recover $Z$ by an approximation $\widetilde{Z}$ and to gain knowledge of the posterior distribution by sampling w.r.t.~$\pi_{\wt Z}$ (which is hopefully close to $\pi_Z$). A theoretical justification of that approach requires a stability investigation of $\pi_Z$ w.r.t.~$Z$.

For quantifying stability properties of probability measures we need to introduce how we want to measure the difference of distributions. 
For this we use either the total variation or the $1$-Wasserstein distance based on the metric $d$ on $(\Theta,\mathcal{B}(\Theta))$. 
Let us mention here that we only consider the 1-Wasserstein distance and therefore will simply refer to it as the Wasserstein distance. 
The main results of this note, stated and proven in Section~\ref{sec: stab_res}, are upper bounds on the difference of $\pi_Z$ and $\pi_{\wt Z}$ in terms of $Z$ and $\wt Z$ with respect to that distances. We use the following notation. For a measure $\nu$ on $(\Theta,\mathcal{B}(\Theta))$ we denote the $L^p(\nu)$-norm, for $p\geq1$, by $\Vert \cdot \Vert_{\nu,p}$, that is, for measurable $f\colon \Theta \to \mathbb{R}$ we have
$
\Vert f \Vert_{\nu,p} := \left( \int_\Theta \vert f(\theta) \vert^p \nu({\rm d} \theta)\right)^{1/p}. 
$  
In Theorem~\ref{thm: tv_est} and Theorem~\ref{thm: Wass_dist} we prove that the total variation and the Wasserstein distance (w.r.t.~$d$) of $\pi_Z$ and $\pi_{\wt{Z}}$ is smaller than a constant times
\[
\left \Vert \frac{ Z}{ \wt Z} -1 \right \Vert_{\pi_{ Z},2}.
\]

For the total variation distance the result holds also with the $L^1(\pi_Z)$-norm instead of the $L^2(\pi_Z)$-norm on the right-hand side.
In addition to that we provide a number of consequences under some further regularity conditions. For example, if $\Vert \exp(-\Phi)/ Z \Vert_{\mu,p} \leq K$ and $\inf_{\theta\in \Theta} \wt Z(\theta) \geq \ell$  
for some $p\in [1,\infty]$, some $K<\infty$ and $\ell>0$, then
\begin{equation}
\label{eq: Lipsch}
\frac{2 K}{\ell C_{Z}} \left\Vert Z - \wt Z \right\Vert_{\mu,p/(p-1)}
\end{equation}
is again an upper bound of the total variation distance. 
Under some additional moment conditions on $\mu$ and $\Vert \exp(-\Phi)/ \wt Z \Vert_{\mu,p} \leq K$ a similar estimate is verified for the Wasserstein distance, see Corollary~\ref{cor: Wass}.
Note that from \eqref{eq: Lipsch} one can conclude a
local Lipschitz continuity of the mapping $Z\mapsto \pi_Z$ from a subset of $L_{\mu,p/(p-1)}$-functions to the set of probability measures on $(\Theta,\mathcal{B}(\Theta))$.

If $Z$ takes the role of a normalizing constant, as in Example~\ref{ex: Boltzmann} above, an approximation $\wt Z$ of $Z$ by numerical integration is natural.
For this purpose Monte Carlo integration or Monte Carlo estimators, respectively, of $Z$ are a common choice. 
Since Monte Carlo methods yield random approximations $\wt Z$, we also conduct a stability analysis allowing for randomized recovery algorithms of $Z$ leading to random probability measures $\pi_{\wt Z}$.
In that randomized scenario we provide estimates of the expected total variation, see Corollary~\ref{cor: expect_tv}, and the expected Wasserstein distance, see Theorem~\ref{thm: expected_Wass}. The total variation result reads as follows
\begin{align*}
\mathbb{E} \Vert \pi_Z - \pi_{\wt Z(\cdot)} \Vert_{\rm tv} 
& \leq 2 \int_{\Theta} \Big( \mathbb{E}\Big \vert \frac{\wt Z(\cdot,\theta)}{Z(\theta)}-1 \Big \vert^2 \Big)^{1/2} \Big(\mathbb{E}\Big [ \frac{Z(\theta)}{\wt Z(\cdot,\theta)} \Big]^2\Big)^{1/2} \pi_Z(\dint \theta).
\end{align*}
Thus, the upper estimate depends on an averaged relative second moment difference of $\wt Z$ and $Z$. 
We apply our randomized stability results to Monte Carlo approximations of $Z$ in two particular examples including the Gibbs distribution of Example~\ref{ex: Boltzmann}.

In the following we discuss how our results fit into the literature. The study of stability properties w.r.t.~posterior distributions in Bayesian inference attracted in recent years considerable attention, see e.g. \cite{dashtistuart17,latz2019well,sprungk2020local,St10}. In the work \cite{sprungk2020local} local Lipschitz continuity for bounded likelihood functions   $\theta \mapsto \exp(-\Phi(\theta))/Z(\theta)$ has been investigated and in   \cite{latz2019well} continuity results of posterior distributions w.r.t.~perturbations within the observed data are proven. In contrast to \cite{sprungk2020local} a consequence of our main estimate is local Lipschitz continuity also for possibly unbounded likelihood functions and in contrast to the continuity study in \cite{latz2019well} we focus on quantitative rather than qualitative results. 
In Bayesian statistics a number of Markov chain approaches have been developed for approximate sampling of $\pi_Z$, for a comprehensive review we refer to \cite{park2018bayesian}. One can distinguish two different types of approaches. The exact one, see \cite{murray22amp,moller2006efficient}, where a Markov chain with limit distribution $\pi_Z$ is constructed and the inexact one, where whenever a function evaluation of $Z$ is needed, an approximation of it is used. The inexact setting leads to Markov chains which not necessarily target $\pi_Z$, but another distribution that is (hopefully) close to $\pi_Z$. In particular, the noisy Markov chain approach, for example investigated in \cite{alquier2016noisy,medina2019perturbation,rudolf2018perturbation} falls into this category. 
In addition to that also adaptive Markov chain Monte Carlo approaches have been developed, see \cite{atchade2013bayesian,liang2016adaptive}. The inexact setting is closely related to our work, since there for a given $\theta\in \Theta$ Monte Carlo approximations of $Z(\theta)$ are employed. 
A similar estimate as the one w.r.t. the expected total variation distance follows from \cite[Theorem~3.2]{lie2018random} and \cite[Theorem~1]{lie2019error} under slightly different boundedness assumptions on the (random) likelihood. However, the estimate of the expected Wasserstein distance seems to be new.


The outline of our work is as follows: 
In the next section we state and prove our stability results. We introduce the total variation and Wasserstein distance as well as defining related quantities which we need for the formulation of our results. Furthermore, for a random function $\wt Z$ 
we provide estimates on the expected total variation and Wasserstein distance.
Finally, we illustrate our bounds in a simple Monte Carlo recovery scenario and in a Gibbs distribution posterior setting. In the latter we use a multiple importance sampling approach.

\section{Stability results}\label{sec: stab_res}
First, we derive bounds of $\pi_Z$ and $\pi_{\wt Z}$ in the total variation and the Wasserstein distance.
After that we state stability results for the more general case of $\wt Z(\theta)$ being a random variable for any $\theta\in \Theta$.

\subsection{Total variation distance}
Given two probability measures $\nu_1,\nu_2$ on $(\Theta, \mathcal{B}(\Theta))$ we define the total variation distance of $\nu_1$ and $\nu_2$ by

\[
\Vert \nu_1 - \nu_2 \Vert_{\rm tv} := \sup_{\vert f \vert_{\infty} \leq 1} \left \vert \mathbb{E}_{\nu_1}(f) - \mathbb{E}_{\nu_2}(f) \right \vert
\]
where $\vert f \vert_{\infty} := \sup_{\theta\in\Theta} \vert f(\theta) \vert$ and $\mathbb{E}_{\nu_1}(f):= \int_\Theta f(\theta) \nu_1(\dint \theta)$ for measurable $f\colon \Theta \to \mathbb{R}$.
\begin{theorem}  \label{thm: tv_est}
	Suppose that $Z\colon \Theta \to (0,\infty)$ and $\widetilde{Z} \colon \Theta \to (0,\infty)$ are measurable functions. Then, for $\pi_Z$ and $\pi_{\widetilde Z}$ as in \eqref{eq: pi_Z}, we have
	\begin{equation}
	\label{eq: tv_est}
	\Vert \pi_Z - \pi_{\widetilde{Z}} \Vert_{{\rm tv}} \leq 2 \left \Vert \frac{ Z}{ \wt Z} -1 \right \Vert_{\pi_{ Z},1}.
	\end{equation}
\end{theorem}
\begin{proof}
	For a measurable function $f\colon \Theta \to \mathbb{R}$ one has
	\begin{align}
	\label{al: representation_diff_means}
	&\quad	\vert\mathbb{E}_{\pi_Z}(f) - \mathbb{E}_{\pi_{\widetilde{Z}}}(f)\vert
	= \left \vert \int_\Theta f(\theta)\exp(-\Phi(\theta))\left[ \frac{1}{\widetilde{Z}(\theta)C_{\wt Z}} 
	- 	
	\frac{1}{Z(\theta) C_{Z}} 
	\right] \mu({\rm d}\theta) \right\vert.
	\end{align}
	In order to bound the right-hand side, we note
	\begin{align}
	\label{al: simple_and_crucial_estimate}
	&
	\left\vert \frac{1}{\widetilde{Z}(\theta)C_{\wt Z} } 
	- 	
	\frac{1}{Z(\theta) C_{Z} } 
	\right\vert 
	\leq 
	\frac{\vert C_Z -  C_{\wt Z}  \vert}{C_{\wt Z}C_{Z} \wt Z(\theta) }
	+ \frac{\vert  Z(\theta) - \wt Z(\theta) \vert}{C_{ Z} \wt Z(\theta) Z(\theta)},
	\end{align}
	such that, with finite $\vert f \vert_\infty := \sup_{\theta\in \Theta} \vert f(\theta) \vert$, we obtain
	\begin{align*}
	&\quad	\vert\mathbb{E}_{\pi_Z}(f) - \mathbb{E}_{\pi_{\widetilde{Z}}}(f)\vert
	\leq \vert f \vert_{\infty}  \frac{\vert C_Z- C_{\wt Z}  \vert}{C_{Z}}
	+ \vert f \vert_{\infty} \left\Vert \frac{Z}{\wt Z}-1 \right\Vert_{\pi_{Z},1}.
	\end{align*}
	Furthermore
	\begin{align}
	\label{al: est_Cs}
	\frac{\vert C_{\wt Z}- C_Z  \vert}{C_{ Z}}
	\leq \frac{1}{C_{ Z}} \int_\Theta \exp(-\Phi(\theta))\left\vert 
	\frac{1}{\wt Z(\theta)} - \frac{1}{Z(\theta)}
	\right \vert \mu({\rm d}\theta) 
	= \left \Vert \frac{ Z}{\wt Z} -1 \right \Vert_{\pi_{ Z},1}.
	\end{align}
	By 
	$\Vert \pi_Z - \pi_{\widetilde{Z}} \Vert_{{\rm tv}}
	= \sup_{\vert f \vert_{\infty}\leq 1} \vert\mathbb{E}_{\pi_Z}(f) - \mathbb{E}_{\pi_{\widetilde{Z}}}(f)\vert
	$ the statement follows.
\end{proof}
Let us provide some remarks and consequences.
\begin{remark}
	The previous estimate is not sharp in the following sense. Set $\wt Z = c Z$ for some constant $c>0$, then the left-hand side is zero, but the right-hand side in general not. This deficiency can be easily repaired by using the fact that $\pi_{\wt Z} = \pi_{\wt c\wt Z}$ for any arbitrary constant $\wt c>0$. With this fact Theorem~\ref{thm: tv_est} implies readily that
	\begin{equation}
	\label{eq: conseq_with_inf}
	\Vert \pi_Z - \pi_{\widetilde{Z}} \Vert_{{\rm tv}} \leq 2 \inf_{\wt c>0} \left \Vert \frac{ Z}{\wt c \wt Z} -1 \right \Vert_{\pi_{ Z},1}.
	\end{equation}
	For the new bound if $\wt Z=c Z$ (with $c>0$), the left- and right-hand side are both zero.
	In particular, for the slightly more conservative upper bound using the $L^2(\pi_Z)$-norm instead of the $L^1(\pi_Z)$-norm on the right-hand side of \eqref{eq: conseq_with_inf}, one can derive that
	\[
	\Vert \pi_Z - \pi_{\widetilde{Z}} \Vert_{{\rm tv}} 
	\leq 2 \left \Vert \frac{\|Z/\wt Z\|_{\pi_Z,1}}{\|Z/\wt Z\|^2_{\pi_Z,2}}\ \frac{ Z}{\wt Z} -1 \right \Vert_{\pi_{ Z},2}
	=
	\inf_{\wt c>0} \left \Vert \frac{ Z}{\wt c \wt Z} -1 \right \Vert_{\pi_{ Z},2}.
	\]
\end{remark}
\begin{remark}  \label{rem: rem_min}
	By interchanging the roles of $Z$ and $\wt Z$ in \eqref{eq: tv_est} one easily obtains
	\[
	\Vert \pi_Z - \pi_{\widetilde{Z}} \Vert_{{\rm tv}} \leq 2 
	\min\left\{
	\left \Vert \frac{\wt Z}{Z} -1 \right \Vert_{\pi_{\wt Z},1}, 
	\left \Vert \frac{Z}{\wt Z} -1 \right \Vert_{\pi_{ Z},1}
	\right\}.
	\] 
\end{remark}
In the light of \cite{sprungk2020local} one might ask for local Lipschitz continuity of the mapping $Z\mapsto \pi_Z$.
\begin{corollary}  \label{cor: first_cons}
	Suppose that $\inf_{\theta\in \Theta} \wt Z(\theta) \geq \ell $ for some $\ell>0$, then
	\[
	\Vert \pi_Z - \pi_{\widetilde{Z}} \Vert_{{\rm tv}} 
	\leq \frac{2}{\ell} \left\Vert \wt Z - Z \right\Vert_{\pi_{Z},1}.
	\]
\end{corollary}
Note that in Theorem~\ref{thm: tv_est} and in the previous corollary on the right-hand side the norm itself already depends on $Z$. One might argue that this hides some $Z$-dependence. Under an additional requirement we can remove this dependence by applying H\"older's inequality.  
\begin{corollary}
	Suppose that for $p\in[1,\infty]$ we have 
	$\Vert \exp(-\Phi)/ Z \Vert_{\mu,p} \leq K$ for some $K <\infty$, then
	\[
	\left\Vert \pi_Z - \pi_{\widetilde{Z}} \right\Vert_{{\rm tv}} 
	\leq \frac{2K}{C_{Z}} \left\Vert \frac{Z}{\wt Z} - 1 \right\Vert_{\mu,p/(p-1)}.
	\]
	If additionally  $\inf_{\theta\in \Theta} \wt Z(\theta) \geq \ell $ for some $\ell>0$ then
	\[
	\left\Vert \pi_Z - \pi_{\widetilde{Z}} \right\Vert_{{\rm tv}} 
	\leq \frac{2K}{\ell C_{Z}} \left\Vert Z - \wt Z \right\Vert_{\mu,p/(p-1)}.
	\]
\end{corollary}
The last inequality provides a local Lipschitz continuity w.r.t.~the $L_{p/(p-1)}(\mu)$-norm in contrast to the local Lipschitz continuity w.r.t the $L_1(\pi_Z)$-norm of Corollary~\ref{cor: first_cons}. For $p=\infty$ it is essentially the estimate of \cite[Theorem~8]{sprungk2020local} in our context.
\begin{remark}
	In the light of the final estimate in the previous corollary let us explain in more detail what we mean with local Lipschitz continuity: Define the set $\mathcal{Z}_{\ell,p,K}$ for $\ell>0$, $K<\infty$ and $1\leq p \leq \infty$ by all measurable functions 
	$Z\colon \Theta\to [\ell,\infty)$ with $Z\in L_{p/(p-1)}(\mu)$ and $\Vert \exp(-\Phi)/ Z \Vert_{\mu,p} \leq K$. Then, for any $Z,\wt Z\in \mathcal{Z}_{\ell,p,K}$ there exists a constant $R_Z<\infty$ such that
	\[
	\Vert \pi_Z -\pi_{\widetilde Z} \Vert_{\rm tv} \leq R_Z  \left\Vert Z - \wt Z \right\Vert_{\mu,p/(p-1)}.
	\]
\end{remark}

\subsection{Wasserstein distance}
In the recent years the Wasserstein distance has become a standard tool in applied probability and statistics, see for example \cite{latz2019well,panaretos2019statistical,rudolf2018perturbation,sprungk2020local}.
One advantage of this distance is that it takes topological properties of the metric space $(\Theta,d)$ into account, which provides a certain flexibility. For instance, for $\theta,\widebar{\theta} \in \Theta$ the Wasserstein distance of the Dirac measures $\delta_\theta$ and $\delta_{\widetilde \theta}$ goes to zero when $d(\theta,\widetilde{\theta}) \to 0$.

Let us briefly provide definitions and basic facts. For probability distributions $\nu_1,\nu_2$ on $(\Theta,\mathcal{B}(\Theta))$ the Wasserstein distance of $\nu_1$ and $\nu_2$ is given by
\[
W(\nu_1,\nu_2) := \inf_{\eta \in C(\nu_1,\nu_2)} \int_{\Theta\times \Theta} d(\theta_1,\theta_2)\; \eta(\dint(\theta_1,\theta_2) ),
\]  
where $C(\nu_1,\nu_2)$ denotes the set of couplings of $\nu_1$ and $\nu_2$, that is, a probability measure $\eta$ on $\Theta\times \Theta$ belongs to $C(\nu_1,\nu_2)$ (by definition) if $\eta(A\times \Theta) = \nu_1(A)$ and $\eta(\Theta\times A)= \nu_2(A)$ for any $A\in\mathcal{B}(\Theta)$. For measurable $f\colon \Theta \to \mathbb{R}$ define the Lipschitz semi-norm
\[
\Vert f \Vert_{\rm Lip} := \sup_{\theta_1,\theta_2 \in \Theta,\;\theta_1\neq\theta_2} \frac{\vert f(\theta_1) - f(\theta_2) \vert}{d(\theta_1,\theta_2)}
\]
and note that the Wasserstein distance allows a dual representation. For arbitrary $\theta_0\in \Theta$ the \emph{Kantorovich-Rubinstein duality}, for details see \cite{Vi09}, is given by
\begin{equation}
\label{eq: Kant_Rub_duality}
W(\nu_1,\nu_2) = \sup_{\Vert f \Vert_{\rm Lip}\leq 1, f(\theta_0)=0} \left \vert  \mathbb{E}_{\nu_1}(f) - \mathbb{E}_{\nu_2}(f) \right \vert.
\end{equation}

\begin{remark}
	Note that we only consider the $1$-Wasserstein distance in contrast to the general $p$-Wasserstein distance given by
	\[
	W_p(\nu_1,\nu_2) := \inf_{\eta \in C(\nu_1,\nu_2)} \left[\int_{\Theta\times \Theta} d(\theta_1,\theta_2)^p\; \eta(\dint(\theta_1,\theta_2) )\right]^{1/p}
	\]
	for $p\geq 1$. The reason for that lies in the fact that we essentially use the Kantorovich-Rubinstein duality. For the $p$-Wasserstein distance there is also a dual characterization, see \cite{Vi09}, which is more involved and does not readily allow to employ \eqref{al: simple_and_crucial_estimate} for estimating $W_p$ in the case of $p>1$.
\end{remark}

Furthermore, for a measure $\nu$ on $(\Theta,\mathcal{B}(\Theta))$ and $p\geq 1$ define
\[
\vert \nu \vert^{(p)} := \inf_{\theta_0 \in \Theta} \left( \int_\Theta d(\theta_0,\theta)^p \nu(\dint \theta) \right)^{1/p}.
\]  
Note that if the metric $d$ is bounded, that is, $\sup_{\theta_1, \theta_2\in \Theta } d(\theta_1,\theta_2) \leq D$ for some $D<\infty$, then $\vert \nu \vert^{(p)} \leq D$.
Now we are able to formulate the Wasserstein stability estimate.

\begin{theorem}  \label{thm: Wass_dist}
	Suppose that $Z\colon \Theta \to (0,\infty)$ and $\widetilde{Z} \colon \Theta \to (0,\infty)$ are measurable functions. Then, for $\pi_Z$ and $\pi_{\widetilde Z}$ as in \eqref{eq: pi_Z}, we have
	\begin{align*}
	W(\pi_Z,\pi_{\widetilde{Z}}) 
	\leq & \left\Vert \frac{ Z}{\wt Z}-1 \right\Vert_{\pi_{Z},1} 
	\left \vert \pi_{\wt Z} \right \vert^{(1)} 
	+ \left\Vert \frac{Z}{\wt Z}-1 \right\Vert_{\pi_{{Z}},2}  	\left \vert \pi_{ Z} \right \vert^{(2)} \\
	\leq & \left\Vert \frac{ Z}{\wt Z}-1 \right\Vert_{\pi_{Z},2} \left( 	\left \vert \pi_{\wt Z} \right \vert^{(1)} + 	\left \vert \pi_{ Z} \right \vert^{(2)}	\right) . 
	\end{align*}
	Moreover, if $Z$ and $\tilde Z$ are sufficiently close to each other, that is, $\left\Vert \frac{Z}{\wt Z}-1 \right\Vert_{\pi_{{Z}},2}\leq 1 - \varepsilon$, for an $\varepsilon \in (0,1)$, then
	\begin{equation} \label{eq: with_moment_of_pi_wt_Z}
	W(\pi_Z,\pi_{\widetilde{Z}}) 
	\leq 
	\frac{2}{\varepsilon} \left \vert \pi_{ Z} \right \vert^{(2)}\  \left\Vert \frac{ Z}{\wt Z}-1 \right\Vert_{\pi_{Z},2}.
	\end{equation}
\end{theorem}
\begin{proof}
	For arbitrary $\theta_0\in\Theta$ by \eqref{eq: Kant_Rub_duality} it is sufficient to estimate the right-hand side of \eqref{al: representation_diff_means}. For this we again use \eqref{al: simple_and_crucial_estimate} and obtain
	\begin{equation}
	\label{eq: Wass_I1_I2_ineq}
	W(\pi_Z,\pi_{\wt Z}) \leq I_1 + I_2,
	\end{equation}
	where
	
	\begin{align*}
	I_1 & := \frac{\vert C_Z-C_{\wt Z}\vert}{C_Z}  \sup_{\Vert f \Vert_{\rm Lip}\leq 1, f(\theta_0)=0} \int_\Theta f(\theta) \pi_{\wt Z}(\dint \theta) \\
	I_2 & := \sup_{\Vert f \Vert_{\rm Lip}\leq 1, f(\theta_0)=0} \int_\Theta f(\theta) \frac{\vert Z(\theta)-\wt Z(\theta)\vert}{\wt Z(\theta)} \;\pi_{Z}(\dint \theta).
	\end{align*}
	By \eqref{al: est_Cs} and taking the infimum over $\theta_0 \in \Theta$ we obtain
	$
	I_1 \leq \left \Vert \frac{ Z}{\wt Z} -1 \right \Vert_{\pi_{ Z},1} \vert \pi_{\wt Z}\vert^{(1)}
	$
	and by additionally using the Cauchy--Schwarz inequality we get
	$
	I_2 \leq \left \Vert \frac{ Z}{\wt Z} -1 \right \Vert_{\pi_{ Z},2} \vert \pi_{Z}\vert^{(2)}
	$
	which proves the first assertion.
	The second statement follows easily from the first by
	\[
	\left \vert \pi_{\wt Z} \right \vert^{(1)}
	\leq
	\left \vert \pi_{Z} \right \vert^{(1)} + \left| \left \vert \pi_{\wt Z} \right \vert^{(1)} - \left \vert \pi_{Z} \right \vert^{(1)} \right|
	\leq
	\left \vert \pi_{Z} \right \vert^{(2)} + W(\pi_Z,\pi_{\widetilde{Z}}) 
	\]
	and rearranging the terms accordingly.
\end{proof}

\begin{remark}
	Having an approximation $\wt Z$ of $Z$ in mind, \eqref{eq: with_moment_of_pi_wt_Z} tells us that ``asymptotically'', that is, with sufficient accuracy of the recovery algorithm $\wt Z$, no explicit bound on $\vert \pi_{\wt Z} \vert^{(1)}$ is required.
\end{remark}

As in the total variation distance consideration, by a boundedness assumption and H\"older's inequality we can exchange the $\pi_Z$-dependence on the right-hand side, by an explicit $Z$-dependence.
\begin{corollary}
	\label{cor: Wass}
	Suppose for $p\in[1,\infty]$ that 
	\[
	\Vert \exp(-\Phi)/Z \Vert_{\mu,p} \leq K 
	\quad \text{and} \quad
	\Vert \exp(-\Phi)/\wt Z \Vert_{\mu,p} \leq K,
	\] for some $K<\infty$, then
	\[
	W(\pi_Z,\pi_{\widetilde{Z}}) \leq \frac{K \vert \mu \vert^{(2p/(p-1))}}{C_Z}\left(\frac{1}{C_Z}+\frac{1}{C_{\wt Z}}\right) \left\Vert \frac{ Z}{\wt Z}-1 \right\Vert_{\mu,2p/(p-1)}.
	\]
	In particular, 
	if additionally $\inf_{\theta\in \Theta} \wt Z(\theta) \geq \ell$ for some $\ell>0$, then
	\[
	W(\pi_Z,\pi_{\widetilde{Z}}) \leq \frac{K \vert \mu \vert^{(2p/(p-1))}}{\ell C_Z}\left(\frac{1}{C_Z}+\frac{1}{C_{\wt Z}}\right) \left\Vert Z-\wt Z \right\Vert_{\mu,2p/(p-1)}.
	\]
\end{corollary}
\begin{proof}
	Set $q:= p/(p-1)$ and note that $1/p+1/q=1$. 
	In the following we frequently apply H\"older's inequality. We have
	\begin{align*}
	\left \Vert \frac{ Z}{\wt Z}-1 \right\Vert_{\pi_{Z},2} 
	= \left\Vert \left( \frac{Z}{\wt Z}-1 \right)^2 \cdot \frac{\exp(-\Phi)}{Z C_Z} \right\Vert_{\mu,1}^{1/2}
	\leq \frac{\sqrt{K}}{C_Z} \left\Vert \frac{Z}{\wt Z}-1 \right \Vert_{\mu,2q}. 
	\end{align*}
	Furthermore, note that $\vert \pi_{\wt Z} \vert^{(1)} \leq \vert \pi_{\wt Z} \vert^{(2)}$ and
	\begin{align*}
	\vert \pi_{\wt Z} \vert^{(2)} = \inf_{\theta_0\in\Theta} \left\Vert d(\theta_0,\cdot)^2 \,\frac{\exp(-\Phi)}{\wt Z C_{\wt Z}} \right\Vert_{\mu,1}^{1/2}
	\leq \frac{\sqrt{K}}{C_{\wt Z}} \vert \mu \vert^{(2q)}.
	\end{align*}
	By the same arguments holds $\vert \pi_{Z} \vert^{(2)} \leq  \frac{\sqrt{K}}{C_{Z}} \vert \mu \vert^{(2q)}.$
	Thus, by taking the previous estimate of Theorem~\ref{thm: Wass_dist} into account the proof is concluded.
\end{proof}

\subsection{Randomization}
\label{sec: rand}
Let $(\Omega,\mathcal{F},\mathbb{P})$ be a probability space and let $\mathbb E$ denote the expectation w.r.t.~$\mathbb P$.
Furthermore, let $\wt Z \colon \Omega\times \Theta \to (0,\infty)$ be a jointly measurable function, thus $\wt Z(\cdot,\theta)$ is a random variable
for all $\theta\in \Theta$.
Then, we call $\wt Z$ random function.
By standard arguments the mapping $\omega \mapsto \Vert \pi_Z - \pi_{\wt Z(\omega,\cdot)} \Vert_{\rm tv}$ from $\Omega$ to $\mathbb{R}$ is measurable.
Similarly, if $\vert \pi_Z \vert^{(1)} <\infty$ and  
$\vert \pi_{\wt Z(\omega,\cdot)} \vert^{(1)} <\infty$ for any $\omega\in \Omega$, then also the mapping $\omega \mapsto W(\pi_Z,\pi_{\wt Z(\omega,\cdot)})$ is measurable. 

In this context Theorem~\ref{thm: tv_est} combined with 
a Fubini argument and Cauchy--Schwarz inequality
lead to the following consequence:
\begin{corollary}
	\label{cor: expect_tv}
	Suppose $Z\colon \Theta \to (0,\infty)$ is measurable and let $\wt Z\colon \Omega \times \Theta\to (0,\infty)$ be a random function. 
	Then, for $\pi_Z$ and $\pi_{\wt Z(\omega)}$, $\omega\in\Omega$, given according to \eqref{eq: pi_Z}, we have
	\begin{align}
	\notag
	\mathbb{E} \Vert \pi_Z - \pi_{\wt Z(\cdot)} \Vert_{\rm tv} 
	& \leq 2 \int_\Theta \mathbb{E} \left \vert \frac{Z(\theta)}{\wt Z(\cdot,\theta)} -1 \right \vert \pi_Z({\rm d} \theta)\\
	\label{eq: expected_tv2}	
	& \leq 2 \int_{\Theta} \Big( \mathbb{E}\Big \vert \frac{\wt Z(\cdot,\theta)}{Z(\theta)}-1 \Big \vert^2 \Big)^{1/2} \Big(\mathbb{E}\Big [ \frac{Z(\theta)}{\wt Z(\cdot,\theta)} \Big]^2\Big)^{1/2} \pi_Z(\dint \theta). 
	\end{align}
\end{corollary}

\begin{remark}
	The final estimate \eqref{eq: expected_tv2} follows immediately by the Cauchy--Schwarz inequality. For Monte Carlo recovery approximation $\wt Z(\cdot,\theta)$ of $Z$ it is usually convenient to bound the right-hand side of the second inequality in Corollary \ref{cor: expect_tv}.
	The reason behind is that the term $\mathbb E\left\vert \frac{\wt Z(\cdot,\theta)}{Z(\theta)}-1 \right\vert^2$ relates readily to the relative mean squared error of the method, whereas the ``reversed'' relative mean absolute error $\mathbb E \left \vert \frac{Z(\theta)}{\wt Z(\cdot,\theta)} -1 \right \vert$ is harder to bound directly.
\end{remark}
Using similar arguments as in the proof of Theorem~\ref{thm: Wass_dist} we obtain the following result w.r.t.~the expected Wasserstein distance.

\begin{theorem} \label{thm: expected_Wass}
	Suppose that $Z\colon \Theta \to (0,\infty)$, that $\wt Z \colon \Omega \times \Theta \to (0,\infty)$ is a random function and that for any $\omega\in \Omega$ the probability measure $\pi_{\wt Z(\omega)}$ and $\pi_Z$ are given through \eqref{eq: pi_Z}.
	
	\begin{enumerate}
		\item[(i)] \label{it: 1st_exp_Wass}
		Assume that for $\mathbb{P}$-almost any $\omega\in \Omega$ we have 
		$\vert \pi_{\wt Z(\omega)} \vert^{(1)} \leq R$ 
		for some $R<\infty$ and $\vert \pi_Z \vert^{(1)}<\infty$. 
		Then
		\begin{align}
		\label{al: 1st_est_exp_Wass}
		\mathbb{E} W(\pi_Z,& \pi_{\wt Z(\cdot)})
		\leq  
		\int_\Theta \mathbb{E} \left \vert \frac{Z(\theta)}{\wt Z(\cdot,
			\theta)}-1 \right \vert\, (d(\theta,\theta_0)+R) \; \pi_Z(\dint \theta),
		\end{align}
		for any $\theta_0\in\Theta$. In particular,
		
		\begin{equation*}
		\label{eq: 2nd_est_expl_Wass}
		\mathbb{E} W(\pi_Z,\pi_{\wt Z(\cdot)}) \leq 
		(\vert \pi_Z\vert^{(2)} + R) \left(
		\int_{\Theta} \mathbb{E} \left \vert \frac{\wt Z(\cdot,\theta)}{Z(\theta)} -1 \right \vert^2 \mathbb{E}\left[ \frac{Z(\theta)}{\wt Z(\cdot,\theta)} \right]^2 \pi_Z(\dint \theta)  
		\right)^{1/2}.
		\end{equation*}
		
		\item[(ii)] \label{it: 2nd_exp_Wass}
		Assume that for $\mathbb{P}$-almost any $\omega\in \Omega$ we have $\left \Vert \frac{Z}{\wt Z(\omega)}-1 \right \Vert_{\pi_Z,2} \leq 1-\varepsilon$ for a number $\varepsilon\in (0,1)$. Then
		\begin{align*}
		\mathbb{E} W(\pi_Z,\pi_{\wt Z(\cdot)})
		\leq \frac{2}{\varepsilon} \vert \pi_Z \vert^{(2)} \Big(\int_{\Theta} \mathbb{E} \Big \vert \frac{Z(\theta)}{\wt Z(\cdot,\theta)}-1 \Big \vert^2 \pi_Z(\dint \theta) \Big)^{1/2}. 
		\end{align*}
	\end{enumerate}
\end{theorem}
\begin{proof}
	First we verify the statement of \eqref{it: 1st_exp_Wass}.
	From \eqref{eq: Wass_I1_I2_ineq} we obtain for arbitrary $\theta_0\in\Theta$ that
	
	\begin{align*}
	\mathbb{E} W(\pi_Z,\pi_{\wt Z(\cdot)}) 
	&	\leq R\, \mathbb{E} \bigg(\frac{\vert C_Z - C_{\wt Z(\cdot)}\vert}{C_Z}  \bigg)
	+ \mathbb{E} \bigg( \int_\Theta d(\theta,\theta_0) \Big \vert \frac{Z(\theta)}{\wt Z(\cdot,
		\theta)}-1 \Big \vert \pi_Z(\dint \theta)\bigg)\\
	& \underset{\eqref{al: est_Cs}}{\leq} R \; \mathbb{E} \Big \Vert \frac{Z}{\wt Z(\cdot)} -1 \Big\Vert_{\pi_Z,1} 
	+ \int_\Theta d(\theta,\theta_0) \, \mathbb{E} \Big \vert \frac{Z(\theta)}{\wt Z(\cdot,
		\theta)}-1 \Big \vert \pi_Z(\dint \theta)\\
	& = \int_\Theta (d(\theta,\theta_0)+R) \;\mathbb{E} \Big \vert \frac{Z(\theta)}{\wt Z(\cdot,
		\theta)}-1 \Big \vert \pi_Z(\dint \theta),
	\end{align*}
	which gives \eqref{al: 1st_est_exp_Wass}. Further estimating (by using the Cauchy--Schwarz inequality several times) leads to
	
	\begin{align*}
	\mathbb{E} W(\pi_Z,\pi_{\wt Z(\cdot)}) 
	& \leq  (\vert \pi_Z \vert^{(2)}+R) 
	\left( \int_\Theta \Big[ \mathbb{E}\Big\vert \frac{Z(\theta)}{\wt Z(\cdot,\theta)} -1\Big\vert \Big]^2 \pi_Z(\dint \theta)\right)^{1/2}\\
	& \leq (\vert \pi_Z \vert^{(2)}+R) 
	\left( \int_\Theta \mathbb{E} \Big \vert \frac{\wt Z(\cdot,\theta)}{Z(\theta)} -1 \Big \vert^2 \mathbb{E}\Big [ \frac{Z(\theta)}{\wt Z(\cdot,\theta)} \Big]^2 \pi_Z(\dint \theta)\right)^{1/2}.
	\end{align*}
	Now we turn to the proof of the statement of \eqref{it: 2nd_exp_Wass}. By \eqref{eq: with_moment_of_pi_wt_Z} it is sufficient to estimate $\mathbb{E} \Vert \frac{Z}{\wt Z(\cdot)}-1 \Vert_{\pi_Z,2}$. With Jensen's inequality as well as a change of integration we have
	\[
	\mathbb{E} \Big\Vert \frac{Z}{\wt Z(\cdot)}-1 \Big \Vert_{\pi_Z,2} 
	\leq \left(\int_{\Theta} \mathbb{E} \Big \vert \frac{Z(\theta)}{\wt Z(\cdot,\theta)} -1 \Big \vert^2 \pi_Z(\dint \theta) \right)^{1/2}
	\]
	which concludes the proof.
\end{proof}
\begin{remark}
	\label{rem: restr_R}
	The condition in the first statement of Theorem~\ref{thm: expected_Wass} that for some $R<\infty$ we assume $\vert \pi_{\wt Z(\omega)} \vert^{(1)} \leq R$ for almost any $\omega\in \Omega$ seems to be restrictive. 
	We would like to comment on that:
	
	\begin{enumerate}
		\item
		If $(\Theta,d)$ is a bounded metric space, then $\vert \pi_{\wt Z(\omega)} \vert^{(1)} \leq R$ holds 
		for $R$ being the upper bound for $d$.
		Further note, that by $\wt d(\theta_1,\theta_2) := \min\{R, d(\theta_1,\theta_2)\}$ we can turn any metric space $(\Theta,d)$ into a topologically equivalent bounded metric space $(\Theta,\wt d)$.
		\item
		If $\wt Z$ is a Monte Carlo approximation of $Z$, then we might be able to derive functions $\wt \ell,\wt u \colon \Theta \to \mathbb{R}$ satisfying $\wt \ell(\theta) \leq \wt Z(\omega,\theta) \leq \wt u(\theta)$ for all $(\omega,\theta)\in \Omega\times \Theta$, cf.~Section \ref{sec:MC}. Setting
		\[
		R := \frac{\inf_{\theta_0\in\Theta} \int_\Theta d(\theta_0,\theta)\frac{\exp(-\Phi(\theta))}{\wt \ell(\theta)} \mu(\dint \theta)}{\int_\Theta \frac{\exp(-\Phi(\theta))}{
				\wt u(\theta)} \mu(\dint \theta)}
		\]
		and assuming its finiteness then leads to $\vert \pi_{\wt Z(\omega)} \vert^{(1)} \leq R$.
		\item
		In addition to that, the second statement of Theorem~\ref{thm: expected_Wass} tells us, within the setting of a Monte Carlo approximation $\wt Z$, that ``asymptotically'', if at some accuracy level $\wt Z$ is sufficiently close to $Z$ with probability one, then no explicit upper bound on $\vert \pi_{\wt Z(\omega)} \vert^{(1)}$ is required.
	\end{enumerate}
\end{remark}
\begin{remark}
	Related to results in this subsection are the recent publications \cite{lie2018random,lie2019error}.
	There, the authors studied well-posedness properties of Bayesian inverse problems with random likelihoods. 
	In particular, they bounded the expected squared Hellinger distance between a random approximate posterior distribution and a desired one. 
\end{remark}

\section{Monte Carlo recovery -- illustrative examples}\label{sec:MC}

Given the probability measure $\pi_Z$ on $(\Theta,\mathcal{B}(\Theta))$ one might ask how to get $\wt Z$ which is ``close'' to $Z$
in the sense of our stability results. Thus, we are looking for an approximation $\wt Z$ of $Z$. 
Observe that, for an arbitrary, maybe even unknown, constant $c>0$ it is actually sufficient to approximate the function $c\cdot Z$, since $\pi_{cZ}=\pi_Z$. 
However, the crucial difficulty lies in the fact that we do not have access to function evaluations of $c\cdot Z$. 
Nevertheless, there exist scenarios where Monte Carlo estimators can successfully be applied, see \cite{alquier2016noisy,atchade2013bayesian,habeck2014bayesian,rudolf2018perturbation}. This motivates the consideration of Monte Carlo recovery methods for $Z$. 
A random function $\wt Z_N$, as defined in Section~\ref{sec: rand}, is a Monte Carlo approximation of $Z$ with information parameter $N$ if $\wt Z_N$ uses at most $N$ pieces of available information, e.g.~function evaluations or samples w.r.t.~a certain distribution. We consider two illustrating scenarios for which we provide explicit Monte Carlo approximations of $Z$.

\subsection{Simple Monte Carlo recovery}
We consider the same framework as in \cite[Section~4.1]{medina2019perturbation}.
Let $(G,\mathcal{G})$ be a measurable space and $\rho\colon G\times \Theta \to[0,\infty)$ be a measurable function, such that
\[
Z(\theta) := \int_G \rho(x,\theta) \, \nu_\theta(\dint x) \in (0,\infty),
\] 
where $(\nu_\theta)_{\theta\in\Theta}$ is a family of probability distribution on $G$. With another measurable function $\Phi\colon \Theta\to (-\infty,\infty)$ this defines the distribution of interest $\pi_Z$ as given in \eqref{eq: pi_Z}.  

We assume that for any $\theta\in \Theta$ we can sample w.r.t.~$\nu_\theta$. Then
$
\wt Z_N(\theta) := \frac{1}{N} \sum_{j=1}^N \rho(X_j^{(\theta)},\theta),
$
with $X_1^{(\theta)},\dots,X_N^{(\theta)}$ being an iid sample w.r.t.~$\nu_\theta$, provides a Monte Carlo approximation of $Z$. 
Define $Q_N(\theta) := Z_N(\theta)/Z(\theta)$ and note that
$
\left(\mathbb{E}\vert Q_N(\theta)-1\vert^2\right)^{1/2} = \frac{\left(\mathbb{E}\vert Q_1(\theta)-1 \vert^2\right)^{1/2}}{\sqrt{N}},
$
as well as
by \cite[Lemma~23]{medina2019perturbation} we have for the second inverse moment of $Q_N(\theta)$ that $\left(\mathbb{E}Q_N(\theta)^{-2}\right)^{1/2} \leq \left(\mathbb{E}Q_1(\theta)^{-2}\right)^{1/2}$. Hence by Corollary~\ref{cor: expect_tv} we obtain
\begin{equation} 
\label{eq: Ex1_tv}
\mathbb{E} \Vert \pi_Z - \pi_{\wt Z_N(\cdot)} \Vert_{\rm tv} 
\leq \frac{2}{\sqrt{N}}\int_\Theta \left(\mathbb{E}\vert Q_1(\theta)-1 \vert^2\cdot  \mathbb{E}Q_1(\theta)^{-2}\right)^{1/2} \pi_Z(\dint \theta).
\end{equation}
To further estimate the former inequality we impose the following regularity assumption throughout the rest of this section.
\begin{assumption}
	For $\rho\colon G\times \Theta \to[0,\infty)$ assume that there are measurable functions $\ell,u \colon \Theta \to (0,\infty)$ satisfying
	$
	\ell(\theta) \leq \rho(x,\theta) \leq u(\theta), \quad \text{for all}\quad (x,\theta)\in G\times \Theta.
	$
\end{assumption}
\noindent
Then, we have
$
\mathbb{E}\vert Q_1(\theta)-1\vert^2 
\leq \mathbb{E} Q_1(\theta)^2 
\leq u(\theta)^2/Z(\theta)^2
$
as well as $\mathbb{E}Q_1(\theta)^{-2} \leq  Z(\theta)^2/\ell(\theta)^2$. Thus, with \eqref{eq: Ex1_tv} we obtain
\begin{equation}
\label{eq: Ex1_tv_ul}
\mathbb{E} \Vert \pi_Z - \pi_{\wt Z_N(\cdot)} \Vert_{\rm tv} 
\leq \frac{2}{\sqrt{N}}\; \left\Vert \frac{u}{\ell} \right\Vert_{\pi_Z,1}.
\end{equation}
Furthermore, setting 

\[
R := \frac{\inf_{\theta_0\in\Theta} \int_\Theta d(\theta_0,\theta)\frac{\exp(-\Phi(\theta))}{\ell(\theta)} \mu(\dint \theta)}{\int_\Theta \frac{\exp(-\Phi(\theta))}{u(\theta)} \mu(\dint \theta)}
\]
and assuming that it is finite easily gives $\vert \pi_{\wt Z_N(\omega,\cdot)} \vert^{(1)} \leq R$ for any $\omega\in \Omega$, see also Remark~\ref{rem: restr_R}. Then, by Theorem~\ref{thm: expected_Wass}
we have

\begin{align}
\mathbb{E} W(\pi_Z,\pi_{\wt Z_N(\cdot)})
\leq & \frac{(R+\vert \pi_Z \vert^{(2)})}{\sqrt{N}} \left\Vert \frac{u}{\ell} \right\Vert_{\pi_Z,2}.
\label{eq: Ex1_Wass}
\end{align}
Summarized, one can say that if the integrals on the right-hand sides of \eqref{eq: Ex1_tv}, \eqref{eq: Ex1_tv_ul} and \eqref{eq: Ex1_Wass} are finite, then the difference of $\pi_Z$ and $\pi_{\wt Z_N(\cdot)}$ measured either in the total variation or Wasserstein distances decreases with the classical Monte Carlo rate of convergence.

\subsection{Gibbs distribution}
\noindent
Let $G$ be a finite set and suppose that there is a measurable function $h\colon G\times \Theta \to (-\infty,\infty)$, such that

\[
\rho(x\mid \theta) := \frac{\exp(-h(x,\theta))}{Z(\theta)},\quad x\in G,	
\]
is a probability mass function on $G$, where $Z(\theta) := \sum_{x\in G} \exp(-h(x,\theta))$ denotes the normalizing constant of $\exp(-h(x,\theta))$ given $\theta\in \Theta$. In that setting, for a given $x_{\rm obs}\in G$ we have
$
\Phi(\theta)=h(x_{\rm obs},\theta).
$
Note that this framework contains Example~\ref{ex: Boltzmann} as a special case. 
We impose the following condition.
\begin{assumption}
	We assume that we can sample on $G$ w.r.t.~the distribution determined by $\rho(\cdot\mid \theta)$ for any $\theta\in\Theta$. 
\end{assumption}	

Usually there are two arguments why this assumption is not too restrictive. The first is that one might rely on perfect sampling and the second, that one can, at least approximately, sample from such distributions by using Markov chains. 
\begin{remark}
	In the scenario of Example~\ref{ex: Boltzmann} for a given $\theta\in \Theta$ (or rather $\beta\in(0,\infty)$) in a series of papers a Monte Carlo product estimator for $Z(\theta)$ (or better $Z(\beta)$) has been analyzed, see \cite{kolmogorovGibbs2018} and the references therein. As an alternative also multiple importance sampling approaches have been applied for the approximation of $Z(\theta)$, see \cite{atchade2013bayesian,habeck2014bayesian,liang2016adaptive}.	
\end{remark}
Motivated by that we consider a simple multiple importance sampling method. Assume that $\theta_1,\dots,\theta_J\in \Theta$ and note that for any $j\in\{1,\dots,J\}$ holds
\[
\frac{Z(\theta)}{Z(\theta_{j})}
= \sum_{x\in G}
\frac{\exp(-(h(x,\theta))}{\exp(-h(x,\theta_{j}))}\,
\rho(x\mid \theta_{j}).
\]
For $N\in\mathbb{N}$ let $X_{1,j},\dots,X_{N,j}$ be an iid sample w.r.t.~$\rho(\cdot\mid \theta_{j})$ and for any $i\in\{1,\dots,N\}$ let $X_{i,1},\dots,X_{i,J}$ be independent. Given  $p_1,\dots,p_J\in(0,1)$ with $\sum_{j=1}^Jp_j=1$ let 

\[
S^{(i)}(\theta) := \sum_{j=1}^J p_j\,\frac{\exp(-h(X_{i,j},\theta))}{\exp(-h(X_{i,j},\theta_{j}))}
\]
and with $S(\theta):=Z(\theta)\sum_{j=1}^J p_j/Z(\theta_{j})$ note that $\mathbb{E} S^{(i)}(\theta) = S(\theta)$. Intuitively $p_j$ weights the importance of the sample derived by $\rho(\cdot \mid \theta_j)$. Then, an unbiased estimator of $S(\theta)$ is given by the multiple importance sampler 

\[
\wt S_{N}(\theta) := \frac{1}{N} \sum_{i=1}^N S^{(i)}(\theta).
\]
Note that $S^{(1)}(\theta),\dots,S^{(N)}(\theta)$ is a sequence of iid random variables. 	With $Q_N(\theta):= \wt S_{N}(\theta)/S(\theta)$ a simple calculation reveals that
$
\left(\mathbb{E}\vert Q_N(\theta)-1\vert^2\right)^{1/2}
= 
\frac{\left(\mathbb{E}\vert Q_1(\theta)-1 \vert^2\right)^{1/2}}{\sqrt{N}}.
$
For the inverse second moment we have by \cite[Lemma~2.3(i)]{medina2019perturbation} that $\mathbb{E}[Q_N(\theta)^{-2}]^{1/2} \leq  	\mathbb{E}[Q_1(\theta)^{-2}]^{1/2}.$
Hence, by the fact that $\pi_Z = \pi_S$ and by Corollary~\ref{cor: first_cons} we obtain 
\begin{align}
\mathbb{E} \Vert \pi_Z - \pi_{\wt S_N(\cdot)} \Vert_{\rm tv} 
\notag
& = \mathbb{E} \Vert \pi_S - \pi_{\wt S_N(\cdot)} \Vert_{\rm tv}\\
& \leq \frac{2}{\sqrt{N}}\int_\Theta \left(\mathbb{E}\vert Q_1(\theta)-1 \vert^2\cdot  \mathbb{E}Q_1(\theta)^{-2}\right)^{1/2} \pi_Z(\dint \theta).
\label{eq: Ex2_tv}
\end{align}
To elaborate on this we impose the following assumption.
\begin{assumption}
	\label{ass: reg_Gibbs}
	For $h\colon G \times \Theta \to (-\infty,\infty)$ assume that there are measurable functions $\ell,u\colon \Theta \to [0,\infty)$ satisfying \quad
	$
	\ell(\theta) \leq \exp(-h(x,\theta)) \leq u(\theta)
	\quad \text{for all} \quad
	(x,\theta)\in G\times \Theta.
	$
\end{assumption}
\noindent
Under this condition we have with the identity of Bieanym{\'e} 
that
\begin{align*}
\mathbb{E} \vert Q_1(\theta)-1\vert^2 
& = \frac{1}{S(\theta)^2} \mathbb{E} \Big[ \sum_{j=1}^J p_j \Big( \frac{\exp(-h(X_{1,j},\theta))}{\exp(-h(X_{1,j},\theta_j))}-\frac{Z(\theta)}{Z(\theta_j)} \Big) \Big]^2 \\
& = \frac{1}{S(\theta)^2} \sum_{j=1}^J p_j^2\; \mathbb{E} \Big[  \frac{\exp(-h(X_{1,j},\theta))}{\exp(-h(X_{1,j},\theta_j))}-\frac{Z(\theta)}{Z(\theta_j)} \Big]^2\\
& \leq \frac{1}{S(\theta)^2} \sum_{j=1}^J p_j^2\; \mathbb{E} \Big( \frac{\exp(-2h(X_{1,j},\theta))}{\exp(-2h(X_{1,j},\theta_j))} \Big)
\leq \frac{u(\theta)^2}{S(\theta)^2} \sum_{j=1}^J \frac{p_j^2}{\ell(\theta_j)^2}
\end{align*}
and
\begin{align*}
\mathbb{E}Q_1(\theta)^{-2}
= S(\theta)^2\; \mathbb{E}\Big( \sum_{k=1}^J p_k \frac{\exp(-h(X_{1,k},\theta))}{\exp(-h(X_{1,k},\theta_k))}  \Big)^{-2}
\leq \frac{S(\theta)^2}{\ell(\theta)^2} \Big( \sum_{k=1}^J \frac{p_k}{u(\theta_k)} \Big)^{-2}.
\end{align*}
Therefore, \eqref{eq: Ex2_tv} implies

\[
\mathbb{E} \Vert \pi_Z - \pi_{\wt S_N(\cdot)} \Vert_{\rm tv} 
\leq \frac{2}{\sqrt{N}} \frac{\Big(\sum_{j=1}^J \frac{p_j^2}{\ell(\theta_j)^2}\Big)^{1/2}}{\sum_{k=1}^J\frac{p_k}{u(\theta_k)}} \left \Vert \frac{u}{\ell} \right \Vert_{\pi_Z,1}.
\]
Under Assumption~\ref{ass: reg_Gibbs} we can also derive an upper bound for the expected Wasserstein distance of $\pi_Z$ and $\pi_{\wt S_N(\cdot)}$. For this observe that
$
\ell(\theta) \sum_{j=1}^J \frac{p_j}{u(\theta_j)}
\leq \wt S_N(\theta) \leq u(\theta) \sum_{k=1}^J \frac{p_k}{\ell(\theta_k)},
$
such that, if 

\[
R := 
\frac{\sum_{k=1}^J \frac{p_k}{\ell(\theta_k)}}{\sum_{j=1}^J \frac{p_j}{u(\theta_j)}} \cdot 
\frac{\inf_{\theta_0\in \Theta} \int_\Theta d(\theta,\theta_0) \frac{u(\theta)}{\ell(\theta)} \mu(\dint \theta)}{\int_\Theta  \frac{u(\theta)}{\ell(\theta)} \mu(\dint \theta)} 
\]
is finite, then $\vert \pi_{\wt S_N(\omega)} \vert^{(1)} \leq R<\infty$, compare Remark~\ref{rem: restr_R}. Now, we obtain with Theorem~\ref{thm: expected_Wass}, in particular \eqref{al: 1st_est_exp_Wass},
that

\[
\mathbb{E}W(\pi_Z,\pi_{\wt S_N(\cdot)})
\leq \frac{(\vert \pi_Z \vert^{(2)}+R)}{\sqrt{N}} \frac{\Big(\sum_{j=1}^J \frac{p_j^2}{\ell(\theta_j)^2}\Big)^{1/2}}{\sum_{k=1}^J\frac{p_k}{u(\theta_k)}}
\left \Vert \frac{u}{\ell} \right\Vert_{\pi_Z,2}.
\]
Summarized, we observe again that the expected total variation as well as expected Wasserstein distance decays with (at least) the usual Monte Carlo rate of convergence.

\section{Conclusion}
We conducted a stability analysis of doubly-intractable distributions. In particular, given two functions $Z,\wt Z\colon\Theta \to (0,\infty) $ we derived estimates of the total variation and Wasserstein distance of $\pi_Z$ and $\wt \pi_{\wt Z}$. Essentially it turns out that if a relative difference between $Z$ and $\wt Z$, measured in a certain $L^p$-sense, is small, then also $\pi_Z$ and $\pi_{\wt Z}$ are close to each other. We also consider a randomization of $\wt Z$, that is, for any $\theta\in \Theta$ we have a random variable $\wt Z(\theta)$. In this context we provide estimates on the expected total variation and Wasserstein distance of $\pi_Z$ to the random measure $\pi_{\wt Z}$. In addition to that we illustrate our bounds in two simple Monte Carlo recovery settings. 

Finally let us comment on further aspects. In the stability analysis we focused on the total variation and Wasserstein distance, but of course also other quantities for measuring the difference of distributions, such as the Hellinger distance or Kullback-Leibler divergence, are reasonable to investigate. Furthermore, in Section~\ref{sec: rand} we only considered the expected difference of distributions. In the light of \cite{kolmogorovGibbs2018} and also \cite{kunsch2019solvable,kunsch2019optimal} statements of the type ``small error with high probability'' are desirable. In particular, an investigation concerning the approximation of functions by Monte Carlo recovery algorithms seems to be a challenging and very interesting task.

%

\providecommand{\bysame}{\leavevmode\hbox to3em{\hrulefill}\thinspace}
\providecommand{\MR}{\relax\ifhmode\unskip\space\fi MR }
\providecommand{\MRhref}[2]{%
	\href{http://www.ams.org/mathscinet-getitem?mr=#1}{#2}
}
\providecommand{\href}[2]{#2}

\end{document}